\documentclass[11pt]{article}
\usepackage{amssymb,amsmath,amsfonts}

\newtheorem{theorem}{Theorem}[section]

\newtheorem{corollary}[theorem]{Corollary}

\newtheorem{definition}[theorem]{Definition}

\numberwithin{equation}{section}
\newenvironment{proof}{\par\noindent{\bf Proof.}}{$\square$\par\bigskip}
\input{tcilatex}
\begin{document}

\title{\textbf{The matrix sequence in terms of bi-periodic Fibonacci numbers}%
}
\author{\texttt{Arzu Coskun} and \texttt{Necati Taskara}}
\date{Department of Mathematics, Faculty of Science,\\
Selcuk University, Campus, 42075, Konya - Turkey \\
[0.3cm] \textit{arzucoskun58@gmail.com} and \textit{ntaskara@selcuk.edu.tr}}
\maketitle

\begin{abstract}
In this paper, we define the bi-periodic Fibonacci \textit{matrix sequence}
that represent bi-periodic Fibonacci numbers. Then, we investigate
generating function, Binet formula and summations of bi-periodic Fibonacci
matrix sequence. After that, we say that some behaviours of bi-periodic
Fibonacci numbers also can be obtained by considering properties of this new
matrix sequence. Finally, we express that well-known matrix sequences, such
as Fibonacci, Pell, $k$-Fibonacci matrix sequences are special cases of this
generalized matrix sequence.


\end{abstract}

\section{Introduction and Preliminaries}

\qquad The special number sequences such as Fibonacci, Lucas, Pell,
Jacobsthal, Padovan and Perrin and their properties have been investigated
in many articles and books (see, for example \cite{1,3,4,6,7,9,10,12,13,15}
and the references cited therein). The Fibonacci numbers have attracted the
attention of mathematicians because of their intrinsic theory and
applications. The ratio of two consecutive of these numbers converges to the
Golden section $\alpha =\frac{1+\sqrt{5}}{2}$. It is also clear that the
ratio has so many applications in, specially, Physics, Engineering,
Architecture, etc.\cite{8}.

After the study of Fibonacci numbers started in the beginning of 13.
century, many authors have generalized this sequence in different ways. One
of those generalizations was published in 2009 by Edson et al. in \cite{3}.
In this reference, the authors defined the bi-periodic Fibonacci $\left\{
q_{n}\right\} _{n\in \mathbb{N}}$ sequence%
\begin{equation}
q_{n}=\left\{ 
\begin{array}{c}
aq_{n-1}+q_{n-2},\ \ \text{if }n\text{ is even} \\ 
bq_{n-1}+q_{n-2},\ \ \text{if }n\text{ is odd}\ 
\end{array}%
\right. \ ,  \label{1.1}
\end{equation}%
where $q_{0}=0,\ q_{1}=1$ and $a,b$ are nonzero real numbers.

On the other hand, the matrix sequences have taken so much interest for
different type of numbers (\cite{2,5,11,14,16}). In \cite{5}, the authors
defined $(s,t)$-Pell and $(s,t)$-Pell--Lucas sequences and $(s,t)$-Pell and $%
(s,t)$-Pell--Lucas matrix sequences, also gave their some properties. Yazlik
et al., in \cite{14}, establish generalized $(s,t)$-matrix sequences and
present some important relationships among $(s,t)$-Fibonacci and $(s,t)$%
-Lucas sequences and their matrix sequences. In \cite{16}, Yilmaz and
Taskara defined the matrix sequences of Padovan and Perrin numbers. Then,
they presented the relationships between these matrix sequences.

The goal of this paper is to define the related \textit{matrix sequence} for
bi-periodic Fibonacci numbers as the first time in the literature. Then, it
will be given the generating function, Binet formula and summation formulas
for this new generalized matrix sequence. Thus, some fundamental properties
of bi-periodic Fibonacci numbers can be obtained by taking into account this
generalized matrix sequence and its properties. By using the results in
Sections 2, we have a great opportunity to obtain some new properties over
this matrix sequence.

\section{The matrix representation of bi-periodic Fibonacci numbers}

In this section, we mainly focus on the matrix sequence of bi-periodic
Fibonacci numbers to get some important results. In fact, we also present
the generating function, Binet formula and summations for the matrix
sequence.

Hence, in the following, we firstly define the bi-periodic Fibonacci matrix
sequence.

\begin{definition}
\label{def1} For $n\in \mathbb{N}$ and $a,b\ $nonzero real numbers , the
bi-periodic Fibonacci matrix sequences $\left( \mathcal{F}_{n}\left(
a,b\right) \right) $ are defined by%
\begin{equation}
\mathcal{F}_{n}\left( a,b\right) =\left\{ 
\begin{array}{c}
a\mathcal{F}_{n-1}\left( a,b\right) +\mathcal{F}_{n-2}\left( a,b\right) 
\text{, }\ n\text{ even} \\ 
b\mathcal{F}_{n-1}\left( a,b\right) +\mathcal{F}_{n-2}\left( a,b\right) 
\text{, }\ n\text{ odd}\ 
\end{array}%
\right.  \label{2.1}
\end{equation}
\end{definition}

with initial conditions%
\begin{equation*}
{\mathcal{F}}_{0}\left( a,b\right) =\left( 
\begin{array}{cc}
1 & 0 \\ 
0 & 1%
\end{array}%
\right) ,{\mathcal{F}}_{1}\left( a,b\right) =\left( 
\begin{array}{cc}
b & \frac{b}{a} \\ 
1 & 0%
\end{array}%
\right) .
\end{equation*}

\vskip0.4cm

In Definition $\ref{def1},$ the matrix $\mathcal{F}_{1}$ is analogue to the
Fibonacci $Q$-matrix which exists for Fibonacci numbers.

In the following theorem, we give the $n$th general term of the matrix
sequence in (\ref{2.1}) via bi-periodic Fibonacci numbers.

\begin{theorem}
\label{teo1} For any integer $n\geq 0,$ we have the matrix sequence%
\begin{equation}
\mathcal{F}_{n}\left( a,b\right) =\left( 
\begin{array}{cc}
\left( \frac{b}{a}\right) ^{\varepsilon (n)}q_{n+1} & \frac{b}{a}q_{n} \\ 
q_{n} & \left( \frac{b}{a}\right) ^{\varepsilon (n)}q_{n-1}%
\end{array}%
\right) ,  \label{2.2}
\end{equation}
\end{theorem}

where $\varepsilon (n)$ is partial function which%
\begin{equation*}
\varepsilon (n)=\left\{ 
\begin{array}{c}
1,\ n\text{ odd} \\ 
0,\ n\text{ even}%
\end{array}%
\right. .
\end{equation*}

\begin{proof}
First of all, by considering (\ref{1.1}), we obtain the equalities $q_{2}=a,$
$q_{-1}=q_{1}=1$ and $q_{0}=0$. And so, first and second steps of the
induction is obtained as follows:%
\begin{equation*}
\mathcal{F}_{0}\left( a,b\right) =\left( 
\begin{array}{cc}
1 & 0 \\ 
0 & 1%
\end{array}%
\right) =\left( 
\begin{array}{cc}
q_{1} & \frac{b}{a}q_{0} \\ 
q_{0} & q_{-1}%
\end{array}%
\right) ,
\end{equation*}%
\begin{equation*}
\mathcal{F}_{1}\left( a,b\right) =\left( 
\begin{array}{cc}
b & \frac{b}{a} \\ 
1 & 0%
\end{array}%
\right) =\left( 
\begin{array}{cc}
\frac{b}{a}q_{2} & \frac{b}{a}q_{1} \\ 
q_{1} & \frac{b}{a}q_{0}%
\end{array}%
\right) .
\end{equation*}

Actually, by assuming the equation in (\ref{2.2}) holds for all $n=k\in 
\mathbb{Z}
^{+}$, we can end up the proof if we manage to show that the case also holds
for $n=k+1$: 
\begin{eqnarray*}
\mathcal{F}_{k+1}\left( a,b\right) &=&\left\{ \QATOP{a\mathcal{F}_{k}\left(
a,b\right) +\mathcal{F}_{k-1}\left( a,b\right) \text{, }k+1\text{ even}}{b%
\mathcal{F}_{k}\left( a,b\right) +\mathcal{F}_{k-1}\left( a,b\right) \text{, 
}k+1\text{ odd}}\right. \\
&=&a^{^{\varepsilon (k)}}b_{k}^{1-^{\varepsilon (k)}}\mathcal{F}_{k}\left(
a,b\right) +\mathcal{F}_{k-1}\left( a,b\right) \\
&=&a^{^{\varepsilon (k)}}b^{1-^{\varepsilon (k)}}\left( 
\begin{array}{cc}
\left( \frac{b}{a}\right) ^{\varepsilon (k)}q_{k+1} & \frac{b}{a}q_{k} \\ 
q_{k} & \left( \frac{b}{a}\right) ^{\varepsilon (k)}q_{k-1}%
\end{array}%
\right) \\
&&+\left( 
\begin{array}{cc}
\left( \frac{b}{a}\right) ^{\varepsilon (k-1)}q_{k} & \frac{b}{a}q_{k-1} \\ 
q_{k-1} & \left( \frac{b}{a}\right) ^{\varepsilon (k-1)}q_{k-2}%
\end{array}%
\right) \\
&=&\left\{ 
\begin{array}{c}
\left( 
\begin{array}{cc}
\frac{b}{a}q_{k+2} & \frac{b}{a}q_{k+1} \\ 
q_{k+1} & \frac{b}{a}q_{k}%
\end{array}%
\right) \text{, }k\text{ even} \\ 
\left( 
\begin{array}{cc}
q_{k+2} & \frac{b}{a}q_{k+1} \\ 
q_{k+1} & q_{k}%
\end{array}%
\right) \text{, }k\text{ odd}%
\end{array}%
\right. .
\end{eqnarray*}

By combining this partial function, we obtain%
\begin{equation*}
\mathcal{F}_{k+1}\left( a,b\right) =\left( 
\begin{array}{cc}
\left( \frac{b}{a}\right) ^{\varepsilon (k+1)}q_{k+2} & \frac{b}{a}q_{k+1}
\\ 
q_{k+1} & \left( \frac{b}{a}\right) ^{\varepsilon (k+1)}q_{k}%
\end{array}%
\right) .
\end{equation*}
\end{proof}

\begin{theorem}
\label{teo2}Let $\mathcal{F}_{n}\left( a,b\right) $ be as in (\ref{2.2}).
Then the following equality is valid for all positive integers:%
\begin{equation*}
\det (\mathcal{F}_{n}\left( a,b\right) )=\left( -\frac{b}{a}\right)
^{\varepsilon (n)}
\end{equation*}
\end{theorem}

\begin{proof}
By using the iteration, we can write%
\begin{equation*}
\det (\mathcal{F}_{1}\left( a,b\right) )=\left\vert 
\begin{array}{cc}
b & \frac{b}{a} \\ 
1 & 0%
\end{array}%
\right\vert =-\frac{b}{a},
\end{equation*}%
\begin{equation*}
\det (\mathcal{F}_{2}\left( a,b\right) )=\left\vert 
\begin{array}{cc}
ab+1 & b \\ 
a & 1%
\end{array}%
\right\vert =1,
\end{equation*}%
\begin{equation*}
\det (\mathcal{F}_{3}\left( a,b\right) )=\left\vert 
\begin{array}{cc}
ab^{2}+2b & b^{2}+\frac{b}{a} \\ 
ab+1 & b%
\end{array}%
\right\vert =-\frac{b}{a},
\end{equation*}

respectively. By iterating this procedure, we get%
\begin{equation*}
\det (\mathcal{F}_{n}\left( a,b\right) )=\left\{ 
\begin{array}{c}
-\frac{b}{a},\text{ }n\text{ odd} \\ 
1,\text{ }n\text{ even}%
\end{array}%
\right.
\end{equation*}

which is desired.\ 
\end{proof}

In \cite{3}, the authors obtained the Cassini identity for bi-periodic
Fibonacci numbers. Now, as a different approximation \ and so as a
consequence of Theorem \ref{teo1} and Theorem \ref{teo2}, in the following
corollary, we rewrite this identity. In fact, in the proof of this
corollary, we just compare determinants.

\begin{corollary}
\label{cor1}Cassini identity for bi-periodic Fibonacci sequence can also be
obtained using bi-periodic Fibonacci matrix sequence. That is, by using
Theorem \ref{teo1} and Theorem \ref{teo2}, we can write%
\begin{equation*}
\left( \frac{b}{a}\right) ^{2\varepsilon (n)}q_{n+1}q_{n-1}-\frac{b}{a}%
q_{n}^{2}=\left( -\frac{b}{a}\right) ^{\varepsilon (n)}.
\end{equation*}

Thus, we obtain%
\begin{equation*}
a^{1-\varepsilon (n)}b^{\varepsilon (n)}q_{n+1}q_{n-1}-a^{\varepsilon
(n)}b^{1-\varepsilon (n)}q_{n}^{2}=a\left( -1\right) ^{n}.
\end{equation*}
\end{corollary}

\begin{theorem}
\label{teo3} For bi-periodic Fibonacci matrix sequence, we have the
generating function%
\begin{equation*}
\sum\limits_{i=0}^{\infty }\mathcal{F}_{i}\left( a,b\right) x^{i}=\dfrac{1}{%
1-\left( ab+2\right) x^{2}+x^{4}}\left( 
\begin{array}{cc}
1+bx-x^{2} & \frac{b}{a}x+bx^{2}-\frac{b}{a}x^{3} \\ 
x+ax^{2}-x^{3} & 1-(ab+1)x^{2}+bx^{3}%
\end{array}%
\right) .
\end{equation*}
\end{theorem}

\begin{proof}
Assume that $G(x)$ is the generating function for the sequence $\left\{ 
\mathcal{F}_{n}\right\} _{n\in \mathbb{N}}$. Then, we have%
\begin{equation*}
G\left( x\right) =\sum\limits_{i=0}^{\infty }\mathcal{F}_{i}\left(
a,b\right) x^{i}=\mathcal{F}_{0}\left( a,b\right) +\mathcal{F}_{1}\left(
a,b\right) x+\sum\limits_{i=2}^{\infty }\mathcal{F}_{i}\left( a,b\right)
x^{i}.
\end{equation*}

Note that%
\begin{equation*}
-bxG\left( x\right) =-bx\sum\limits_{i=0}^{\infty }\mathcal{F}_{i}\left(
a,b\right) x^{i}=-bx\mathcal{F}_{0}\left( a,b\right)
-b\sum\limits_{i=2}^{\infty }\mathcal{F}_{i-1}\left( a,b\right) x^{i}
\end{equation*}

and%
\begin{equation*}
-x^{2}G\left( x\right) =-\sum\limits_{i=2}^{\infty }\mathcal{F}_{i-2}\left(
a,b\right) x^{i}.
\end{equation*}

Thus, we can write%
\begin{eqnarray*}
\left( 1-bx-x^{2}\right) G\left( x\right) &=&\mathcal{F}_{0}\left(
a,b\right) +x\left( \mathcal{F}_{1}\left( a,b\right) -b\mathcal{F}_{0}\left(
a,b\right) \right) \\
&&+\sum\limits_{i=2}^{\infty }\left( \mathcal{F}_{i}\left( a,b\right) -b%
\mathcal{F}_{i-1}\left( a,b\right) -\mathcal{F}_{i-2}\left( a,b\right)
\right) x^{i}.
\end{eqnarray*}

Since $\mathcal{F}_{2i+1}\left( a,b\right) =b\mathcal{F}_{2i}\left(
a,b\right) +\mathcal{F}_{2i-1}\left( a,b\right) $, we get%
\begin{eqnarray*}
\left( 1-bx-x^{2}\right) G\left( x\right) &=&\mathcal{F}_{0}\left(
a,b\right) +x\left( \mathcal{F}_{1}\left( a,b\right) -b\mathcal{F}_{0}\left(
a,b\right) \right) \\
&&+\sum\limits_{i=1}^{\infty }\left( \mathcal{F}_{2i}\left( a,b\right) -b%
\mathcal{F}_{2i-1}\left( a,b\right) -\mathcal{F}_{2i-2}\left( a,b\right)
\right) x^{2i} \\
&=&\mathcal{F}_{0}\left( a,b\right) +x\left( \mathcal{F}_{1}\left(
a,b\right) -b\mathcal{F}_{0}\left( a,b\right) \right) \\
&&+\left( a-b\right) x\sum\limits_{i=1}^{\infty }\mathcal{F}_{2i-1}\left(
a,b\right) x^{2i-1}.
\end{eqnarray*}

Now, let%
\begin{equation*}
g(x)=\sum\limits_{i=1}^{\infty }\mathcal{F}_{2i-1}\left( a,b\right) x^{2i-1}.
\end{equation*}

Since%
\begin{eqnarray*}
\mathcal{F}_{2i+1}\left( a,b\right) &=&b\mathcal{F}_{2i}\left( a,b\right) +%
\mathcal{F}_{2i-1}\left( a,b\right) \\
&=&b(a\mathcal{F}_{2i-1}\left( a,b\right) +\mathcal{F}_{2i-2}\left(
a,b\right) )+\mathcal{F}_{2i-1}\left( a,b\right) \\
&=&(ab+1)\mathcal{F}_{2i-1}\left( a,b\right) +b\mathcal{F}_{2i-2}\left(
a,b\right) \\
&=&(ab+1)\mathcal{F}_{2i-1}\left( a,b\right) +\mathcal{F}_{2i-1}\left(
a,b\right) -\mathcal{F}_{2i-3}\left( a,b\right) \\
&=&(ab+2)\mathcal{F}_{2i-1}\left( a,b\right) -\mathcal{F}_{2i-3}\left(
a,b\right) ,
\end{eqnarray*}

we have%
\begin{eqnarray*}
\left( 1-\left( ab+2\right) x^{2}+x^{4}\right) g\left( x\right)  &=&\mathcal{%
F}_{1}\left( a,b\right) x+\mathcal{F}_{3}\left( a,b\right) x^{3}-(ab+2)%
\mathcal{F}_{1}\left( a,b\right) x^{3} \\
&&+\sum\limits_{i=3}^{\infty }\left\{ 
\begin{array}{c}
\mathcal{F}_{2i-1}\left( a,b\right) -(ab+2)\mathcal{F}_{2i-3}\left(
a,b\right)  \\ 
\text{ \ \ \ \ \ }+\mathcal{F}_{2i-5}\left( a,b\right) 
\end{array}%
\right\} x^{2i-1}.
\end{eqnarray*}

Therefore,%
\begin{eqnarray*}
g\left( x\right) &=&\frac{\mathcal{F}_{1}\left( a,b\right) x+\mathcal{F}%
_{3}\left( a,b\right) x^{3}-(ab+2)\mathcal{F}_{1}\left( a,b\right) x^{3}}{%
1-\left( ab+2\right) x^{2}+x^{4}} \\
&=&\frac{\mathcal{F}_{1}\left( a,b\right) x+(b\mathcal{F}_{0}\left(
a,b\right) -\mathcal{F}_{1}\left( a,b\right) )x^{3}}{1-\left( ab+2\right)
x^{2}+x^{4}}
\end{eqnarray*}%
and as a result, we get%
\begin{equation*}
G\left( x\right) =\frac{\left\{ 
\begin{array}{c}
\mathcal{F}_{0}\left( a,b\right) +x\mathcal{F}_{1}\left( a,b\right)
+x^{2}\left( a\mathcal{F}_{1}\left( a,b\right) -\mathcal{F}_{0}\left(
a,b\right) -ab\mathcal{F}_{0}\left( a,b\right) \right) \\ 
\text{ \ \ \ \ \ \ \ \ \ \ \ \ \ }+x^{3}\left( b\mathcal{F}_{0}\left(
a,b\right) -\mathcal{F}_{1}\left( a,b\right) \right)%
\end{array}%
\right\} }{1-\left( ab+2\right) x^{2}+x^{4}}.
\end{equation*}%
which is desired equality.
\end{proof}

\begin{theorem}
\label{teo4} For every $n\in 
\mathbb{N}
,$ we write the Binet formula for the bi-periodic Fibonacci matrix sequence
as the form%
\begin{equation*}
\mathcal{F}_{n}\left( a,b\right) =A_{1}\left( \alpha ^{n}-\beta ^{n}\right)
+B_{1}\left( \alpha ^{2\left\lfloor \frac{n}{2}\right\rfloor +2}-\beta
^{2\left\lfloor \frac{n}{2}\right\rfloor +2}\right) ,
\end{equation*}%
where 
\begin{eqnarray*}
A_{1} &=&\dfrac{\left[ {\mathcal{F}}_{1}\left( a,b\right) -b\mathcal{F}%
_{0}\left( a,b\right) \right] ^{\varepsilon (n)}\left[ a{\mathcal{F}}%
_{1}\left( a,b\right) -\mathcal{F}_{0}\left( a,b\right) -ab\mathcal{F}%
_{0}\left( a,b\right) \right] ^{1-\varepsilon (n)}}{\left( ab\right)
^{\left\lfloor \frac{n}{2}\right\rfloor }\left( \alpha -\beta \right) },~ \\
B_{1} &=&\dfrac{b^{\varepsilon (n)}\mathcal{F}_{0}\left( a,b\right) }{\left(
ab\right) ^{\left\lfloor \frac{n}{2}\right\rfloor +1}\left( \alpha -\beta
\right) },
\end{eqnarray*}%
such that $\alpha =\frac{ab+\sqrt{a^{2}b^{2}+4ab}}{2},\beta =\frac{ab-\sqrt{%
a^{2}b^{2}+4ab}}{2},$ and $\varepsilon (n)=n-2\left\lfloor \frac{n}{2}%
\right\rfloor $.
\end{theorem}

\begin{proof}
We know that the generating function of bi-periodic Fibonacci matrix
sequence is%
\begin{equation*}
G\left( x\right) =\frac{\left\{ 
\begin{array}{c}
\mathcal{F}_{0}\left( a,b\right) +x\mathcal{F}_{1}\left( a,b\right)
+x^{2}\left( a\mathcal{F}_{1}\left( a,b\right) -\mathcal{F}_{0}\left(
a,b\right) -ab\mathcal{F}_{0}\left( a,b\right) \right) \\ 
\text{ \ \ \ \ \ \ \ \ \ \ \ \ \ \ \ }+x^{3}\left( b\mathcal{F}_{0}\left(
a,b\right) -\mathcal{F}_{1}\left( a,b\right) \right)%
\end{array}%
\right\} }{1-\left( ab+2\right) x^{2}+x^{4}}.
\end{equation*}

Using the partial fraction decomposition, we rewrite $G\left( x\right) $ as%
\begin{equation*}
G\left( x\right) =\frac{1}{\alpha -\beta }\left[ 
\begin{array}{c}
\frac{\left\{ 
\begin{array}{c}
x\left\{ \alpha \left( b\mathcal{F}_{0}\left( a,b\right) -\mathcal{F}%
_{1}\left( a,b\right) \right) +b\mathcal{F}_{0}\left( a,b\right) \right\} \\ 
\text{ \ \ }+\alpha \left( a\mathcal{F}_{1}\left( a,b\right) -\mathcal{F}%
_{0}\left( a,b\right) -ab\mathcal{F}_{0}\left( a,b\right) \right) \\ 
\text{ \ \ \ \ \ }+a\mathcal{F}_{1}\left( a,b\right) -ab\mathcal{F}%
_{0}\left( a,b\right)%
\end{array}%
\right\} }{x^{2}-\left( \alpha +1\right) }\text{ \ \ \ \ \ \ \ \ \ \ \ \ \ \
\ \ \ } \\ 
\\ 
+\frac{\left\{ 
\begin{array}{c}
x\left\{ \beta \left( \mathcal{F}_{1}\left( a,b\right) -b\mathcal{F}%
_{0}\left( a,b\right) \right) -b\mathcal{F}_{0}\left( a,b\right) \right\} \\ 
\text{ \ \ \ \ \ }+\beta \left( ab\mathcal{F}_{0}\left( a,b\right) +\mathcal{%
F}_{0}\left( a,b\right) -a\mathcal{F}_{1}\left( a,b\right) \right) \\ 
+ab\mathcal{F}_{0}\left( a,b\right) -a\mathcal{F}_{1}\left( a,b\right)%
\end{array}%
\right\} }{x^{2}-\left( \beta +1\right) }%
\end{array}%
\right] .
\end{equation*}

Since the Maclaurin series expansion of the function $\frac{A-Bx}{x^{2}-C}$
is given by%
\begin{equation*}
\frac{A-Bx}{x^{2}-C}=\overset{\infty }{\underset{n=0}{\sum }}%
BC^{-n-1}x^{2n+1}-\overset{\infty }{\underset{n=0}{\sum }}AC^{-n-1}x^{2n},
\end{equation*}

the generating function $G\left( x\right) $ can be expressed as%
\begin{equation*}
G(x)=\frac{1}{\alpha -\beta }\left\{ 
\begin{array}{c}
\overset{\infty }{\underset{n=0}{\sum }}\frac{\left\{ 
\begin{array}{c}
\left\{ 
\begin{array}{c}
\alpha \left( \mathcal{F}_{1}\left( a,b\right) -b\mathcal{F}_{0}\left(
a,b\right) \right)  \\ 
-b\mathcal{F}_{0}\left( a,b\right) 
\end{array}%
\right\} \left( \beta +1\right) ^{n+1}\text{ \ \ \ \ \ \ \ \ } \\ 
+\left\{ 
\begin{array}{c}
\beta \left( b\mathcal{F}_{0}\left( a,b\right) -\mathcal{F}_{1}\left(
a,b\right) \right)  \\ 
+b\mathcal{F}_{0}\left( a,b\right) 
\end{array}%
\right\} \left( \alpha +1\right) ^{n+1}%
\end{array}%
\right\} }{\left( \alpha +1\right) ^{n+1}\left( \beta +1\right) ^{n+1}}%
x^{2n+1} \\ 
\\ 
-\overset{\infty }{\underset{n=0}{\sum }}\frac{\left\{ 
\begin{array}{c}
\alpha \left( a\mathcal{F}_{1}\left( a,b\right) -\mathcal{F}_{0}\left(
a,b\right) -ab\mathcal{F}_{0}\left( a,b\right) \right)  \\ 
\text{ \ \ \ \ \ \ \ }+a\mathcal{F}_{1}\left( a,b\right) -ab\mathcal{F}%
_{0}\left( a,b\right) 
\end{array}%
\right\} }{\left( \alpha +1\right) ^{n+1}\left( \beta +1\right) ^{n+1}}%
\left( \beta +1\right) ^{n+1}x^{2n} \\ 
\\ 
-\overset{\infty }{\underset{n=0}{\sum }}\frac{\left\{ 
\begin{array}{c}
\beta \left( ab\mathcal{F}_{0}\left( a,b\right) +\mathcal{F}_{0}\left(
a,b\right) -a\mathcal{F}_{1}\left( a,b\right) \right)  \\ 
\text{ \ \ \ \ }+ab\mathcal{F}_{0}\left( a,b\right) -a\mathcal{F}_{1}\left(
a,b\right) 
\end{array}%
\right\} }{\left( \alpha +1\right) ^{n+1}\left( \beta +1\right) ^{n+1}}%
\left( \alpha +1\right) ^{n+1}x^{2n}%
\end{array}%
\right\} .
\end{equation*}

By using properties of $\alpha $ and $\beta $, since we know that $\alpha $
and $\beta $ are roots of equation $X^{2}-abX-ab=0,$ we obtain%
\begin{eqnarray*}
G(x) &=&\frac{1}{\alpha -\beta }\overset{\infty }{\underset{n=0}{\sum }}%
\left( \frac{1}{ab}\right) ^{n+1}\left\{ 
\begin{array}{c}
-ab\left( \mathcal{F}_{1}\left( a,b\right) -b\mathcal{F}_{0}\left(
a,b\right) \right) \beta ^{2n+1}\text{ \ \ \ \ \ } \\ 
-ab\left( b\mathcal{F}_{0}\left( a,b\right) -\mathcal{F}_{1}\left(
a,b\right) \right) \alpha ^{2n+1}\text{ \  \ } \\ 
-b\mathcal{F}_{0}\left( a,b\right) \beta ^{2n+2}+b\mathcal{F}_{0}\left(
a,b\right) \alpha ^{2n+2}%
\end{array}%
\right\} x^{2n+1} \\
&&+\frac{1}{\alpha -\beta }\overset{\infty }{\underset{n=0}{\sum }}\left( 
\frac{1}{ab}\right) ^{n+1}\left\{ 
\begin{array}{c}
-ab\left\{ 
\begin{array}{c}
a\mathcal{F}_{1}\left( a,b\right) -\mathcal{F}_{0}\left( a,b\right)  \\ 
\text{ \ \ \ \ }-ab\mathcal{F}_{0}\left( a,b\right) 
\end{array}%
\right\} \beta ^{2n} \\ 
-ab\left\{ 
\begin{array}{c}
ab\mathcal{F}_{0}\left( a,b\right) +\mathcal{F}_{0}\left( a,b\right)  \\ 
\text{ \ \ \ }-a\mathcal{F}_{1}\left( a,b\right) 
\end{array}%
\right\} \alpha ^{2n} \\ 
-\mathcal{F}_{0}\left( a,b\right) \beta ^{2n+2}+\mathcal{F}_{0}\left(
a,b\right) \alpha ^{2n+2}%
\end{array}%
\right\} x^{2n}.
\end{eqnarray*}%
\begin{eqnarray*}
G(x) &=&\overset{\infty }{\underset{n=0}{\sum }}\left( \frac{1}{ab}\right)
^{n}\frac{\left( \mathcal{F}_{1}\left( a,b\right) -b\mathcal{F}_{0}\left(
a,b\right) \right) \left( \alpha ^{2n+1}-\beta ^{2n+1}\right) }{\alpha
-\beta }x^{2n+1} \\
&&+\overset{\infty }{\underset{n=0}{\sum }}\left( \frac{1}{ab}\right) ^{n+1}%
\frac{b\mathcal{F}_{0}\left( a,b\right) \left( \alpha ^{2n+2}-\beta
^{2n+2}\right) }{\alpha -\beta }x^{2n+1} \\
&&+\sum \left( \frac{1}{ab}\right) ^{n}\frac{\left( ab\mathcal{F}_{0}\left(
a,b\right) +\mathcal{F}_{0}\left( a,b\right) -a\mathcal{F}_{1}\left(
a,b\right) \right) \left( \beta ^{2n}-\alpha ^{2n}\right) }{\alpha -\beta }%
x^{2n} \\
&&\overset{\infty }{\underset{n=0}{+\sum }}\left( \frac{1}{ab}\right) ^{n+1}%
\frac{\mathcal{F}_{0}\left( a,b\right) \left( \alpha ^{2n+2}-\beta
^{2n+2}\right) }{\alpha -\beta }x^{2n}.
\end{eqnarray*}

Combining the sums, we get%
\begin{equation*}
G(x)=\overset{\infty }{\underset{n=0}{\sum }}\left\{ 
\begin{array}{c}
\left[ \mathcal{F}_{1}\left( a,b\right) -b\mathcal{F}_{0}\left( a,b\right) %
\right] ^{\varepsilon (n)}\left\{ 
\begin{array}{c}
a\mathcal{F}_{1}\left( a,b\right) -\mathcal{F}_{0}\left( a,b\right)  \\ 
\text{ \ \ \ \ }-ab\mathcal{F}_{0}\left( a,b\right) 
\end{array}%
\right\} ^{1-\varepsilon (n)}\left( \frac{\alpha ^{n}-\beta ^{n}}{\left(
ab\right) ^{\left\lfloor \frac{n}{2}\right\rfloor }\left( \alpha -\beta
\right) }\right)  \\ 
+b^{\varepsilon (n)}\mathcal{F}_{0}\left( a,b\right) \left( \frac{\alpha
^{2\left( \left\lfloor \frac{n}{2}\right\rfloor +1\right) }-\beta ^{2\left(
\left\lfloor \frac{n}{2}\right\rfloor +1\right) }}{\left( ab\right)
^{\left\lfloor \frac{n}{2}\right\rfloor +1}\left( \alpha -\beta \right) }%
\right) 
\end{array}%
\right\} x^{n}.
\end{equation*}

Therefore, for all $n\geq 0$, from the definition of generating function, we
have%
\begin{equation*}
\mathcal{F}_{n}\left( a,b\right) =A_{1}\left( \alpha ^{n}-\beta ^{n}\right)
+B_{1}\left( \alpha ^{2\left\lfloor \frac{n}{2}\right\rfloor +2}-\beta
^{2\left\lfloor \frac{n}{2}\right\rfloor +2}\right) ,
\end{equation*}

which is desired.
\end{proof}

Now, for bi-periodic Fibonacci matrix sequence, we give the some \textit{%
summations by} considering Binet formula.

\begin{theorem}
\label{teo5} For $k\geq 0$, the following statements are true:
\end{theorem}

\begin{itemize}
\item[$(i)$] 
\begin{equation}
\sum\limits_{k=0}^{n-1}\mathcal{F}_{k}\left( a,b\right) =\dfrac{\left\{ 
\begin{array}{c}
a^{\varepsilon \left( n\right) }b^{1-\varepsilon \left( n\right) }\mathcal{F}%
_{n}\left( a,b\right) +a^{1-\varepsilon \left( n\right) }b^{\varepsilon
\left( n\right) }\mathcal{F}_{n-1}\left( a,b\right) \\ 
\text{ \ \ }-a\mathcal{F}_{1}\left( a,b\right) +ab\mathcal{F}_{0}\left(
a,b\right) -b\mathcal{F}_{0}\left( a,b\right)%
\end{array}%
\right\} }{ab},  \notag
\end{equation}

\item[$(ii)$] 
\begin{equation}
\sum\limits_{k=0}^{n}\mathcal{F}_{k}\left( a,b\right) x^{-k}=\frac{1}{%
1-(ab+2)x^{2}+x^{4}}\left\{ 
\begin{array}{c}
\dfrac{\mathcal{F}_{n-1}\left( a,b\right) }{x^{n-1}}-\dfrac{\mathcal{F}%
_{n+1}\left( a,b\right) }{x^{n-3}} \\ 
+\dfrac{\mathcal{F}_{n}\left( a,b\right) }{x^{n}}-\dfrac{\mathcal{F}%
_{n+2}\left( a,b\right) }{x^{n+2}} \\ 
+x^{4}\mathcal{F}_{0}\left( a,b\right) +x^{3}\mathcal{F}_{1}\left( a,b\right)
\\ 
-x^{2}\left[ \left( ab+1\right) \mathcal{F}_{0}\left( a,b\right) -a\mathcal{F%
}_{1}\left( a,b\right) \right] \\ 
-x\left( \mathcal{F}_{1}\left( a,b\right) -b\mathcal{F}_{0}\left( a,b\right)
\right)%
\end{array}%
\right\} .  \notag
\end{equation}
\end{itemize}

where $\alpha =\frac{ab+\sqrt{a^{2}b^{2}+4ab}}{2},\beta =\frac{ab-\sqrt{%
a^{2}b^{2}+4ab}}{2}$ and $\varepsilon (n)=n-2\left\lfloor \frac{n}{2}%
\right\rfloor $.

\begin{proof}
We omit the proof of $(ii)$, because it can be done similarly as in the
proof of $(i)$. We investigate the situation according to the $n$ is even or
odd. Thus, for even $n$%
\begin{eqnarray*}
\sum\limits_{k=0}^{n-1}\mathcal{F}_{k}\left( a,b\right)
&=&\sum\limits_{k=0}^{\frac{n-2}{2}}\mathcal{F}_{2k}\left( a,b\right)
+\sum\limits_{k=0}^{\frac{n-2}{2}}\mathcal{F}_{2k+1}\left( a,b\right) \\
&=&\sum\limits_{k=0}^{\frac{n-2}{2}}\frac{a\mathcal{F}_{1}\left( a,b\right) -%
\mathcal{F}_{0}\left( a,b\right) -ab\mathcal{F}_{0}\left( a,b\right) }{%
\left( ab\right) ^{k}}\frac{\alpha ^{2k}-\beta ^{2k}}{\alpha -\beta } \\
&&+\sum\limits_{k=0}^{\frac{n-2}{2}}\frac{\mathcal{F}_{0}\left( a,b\right) }{%
\left( ab\right) ^{k+1}}\frac{\alpha ^{2k+2}-\beta ^{2k+2}}{\alpha -\beta }
\\
&&+\sum\limits_{k=0}^{\frac{n-2}{2}}\frac{\mathcal{F}_{1}\left( a,b\right) -b%
\mathcal{F}_{0}\left( a,b\right) }{\left( ab\right) ^{k}}\frac{\alpha
^{2k+1}-\beta ^{2k+1}}{\alpha -\beta } \\
&&+\sum\limits_{k=0}^{\frac{n-2}{2}}\frac{b\mathcal{F}_{0}\left( a,b\right) 
}{\left( ab\right) ^{k+1}}\frac{\alpha ^{2k+2}-\beta ^{2k+2}}{\alpha -\beta }%
.
\end{eqnarray*}

In here, simplifying the last equality, we imply%
\begin{eqnarray*}
\sum\limits_{k=0}^{n-1}\mathcal{F}_{k}\left( a,b\right) &=&\frac{\left\{ 
\begin{array}{c}
a\mathcal{F}_{1}\left( a,b\right) -\mathcal{F}_{0}\left( a,b\right) \\ 
\text{ \ \ \ }-ab\mathcal{F}_{0}\left( a,b\right)%
\end{array}%
\right\} }{\alpha -\beta }\left\{ 
\begin{array}{c}
\frac{\alpha ^{n}-\left( ab\right) ^{\frac{n}{2}}}{\left( ab\right) ^{\frac{n%
}{2}-1}\left( \alpha ^{2}-ab\right) }\text{ \ \ \ \ \ } \\ 
-\frac{\beta ^{n}-\left( ab\right) ^{\frac{n}{2}}}{\left( ab\right) ^{\frac{n%
}{2}-1}\left( \beta ^{2}-ab\right) }%
\end{array}%
\right\} \\
&&+\frac{\mathcal{F}_{0}\left( a,b\right) }{\alpha -\beta }\left\{ 
\begin{array}{c}
\frac{\alpha ^{n+2}-\alpha ^{2}\left( ab\right) ^{\frac{n}{2}}}{\left(
ab\right) ^{\frac{n}{2}}\left( \alpha ^{2}-ab\right) }\text{ \ \ \ \ \ \ }
\\ 
-\frac{\beta ^{n+2}-\beta ^{2}\left( ab\right) ^{\frac{n}{2}}}{\left(
ab\right) ^{\frac{n}{2}}\left( \beta ^{2}-ab\right) }%
\end{array}%
\right\} \\
&&+\frac{\mathcal{F}_{1}\left( a,b\right) -b\mathcal{F}_{0}\left( a,b\right) 
}{\alpha -\beta }\left\{ 
\begin{array}{c}
\frac{\alpha ^{n+1}-\alpha \left( ab\right) ^{\frac{n}{2}}}{\left( ab\right)
^{\frac{n}{2}-1}\left( \alpha ^{2}-ab\right) }\text{ \ \ \ \ } \\ 
-\frac{\beta ^{n+1}-\beta \left( ab\right) ^{\frac{n}{2}}}{\left( ab\right)
^{\frac{n}{2}-1}\left( \beta ^{2}-ab\right) }%
\end{array}%
\right\} \\
&&+\frac{b\mathcal{F}_{0}\left( a,b\right) }{\alpha -\beta }\left\{ 
\begin{array}{c}
\frac{\alpha ^{n+2}-\alpha ^{2}\left( ab\right) ^{\frac{n}{2}}}{\left(
ab\right) ^{\frac{n}{2}}\left( \alpha ^{2}-ab\right) }\text{ \ \ \ \ \ } \\ 
-\frac{\beta ^{n+2}-\beta ^{2}\left( ab\right) ^{\frac{n}{2}}}{\left(
ab\right) ^{\frac{n}{2}}\left( \beta ^{2}-ab\right) }%
\end{array}%
\right\} .
\end{eqnarray*}%
\begin{eqnarray*}
\sum\limits_{k=0}^{n-1}\mathcal{F}_{k}\left( a,b\right) &=&\frac{a\mathcal{F}%
_{1}\left( a,b\right) -\mathcal{F}_{0}\left( a,b\right) -ab\mathcal{F}%
_{0}\left( a,b\right) }{\alpha -\beta }\left\{ 
\begin{array}{c}
-\frac{\alpha ^{n-2}-\beta ^{n-2}}{\left( ab\right) ^{\frac{n}{2}}}\text{ \
\ } \\ 
+\frac{\alpha ^{n}-\beta ^{n}}{\left( ab\right) ^{\frac{n}{2}+1}}\text{ } \\ 
-\frac{\alpha ^{2}-\beta ^{2}}{\left( ab\right) ^{2}}%
\end{array}%
\right\} \\
&&+\frac{\mathcal{F}_{0}\left( a,b\right) }{\alpha -\beta }\left\{ -\frac{%
\alpha ^{n}-\beta ^{n}}{\left( ab\right) ^{\frac{n}{2}+1}}+\frac{\alpha
^{n+2}-\beta ^{n+2}}{\left( ab\right) ^{\frac{n}{2}+2}}-\frac{\alpha
^{2}-\beta ^{2}}{\left( ab\right) ^{2}}\right\} \\
&&+\frac{\mathcal{F}_{1}\left( a,b\right) -b\mathcal{F}_{0}\left( a,b\right) 
}{\alpha -\beta }\left\{ -\frac{\alpha ^{n-1}-\beta ^{n-1}}{\left( ab\right)
^{\frac{n}{2}}}+\frac{\alpha ^{n+1}-\beta ^{n+1}}{\left( ab\right) ^{\frac{n%
}{2}+1}}\right\} \\
&&+\frac{b\mathcal{F}_{0}\left( a,b\right) }{\alpha -\beta }\left\{ -\frac{%
\alpha ^{n}-\beta ^{n}}{\left( ab\right) ^{\frac{n}{2}+1}}+\frac{\alpha
^{n+2}-\beta ^{n+2}}{\left( ab\right) ^{\frac{n}{2}+2}}-\frac{\alpha
^{2}-\beta ^{2}}{\left( ab\right) ^{2}}\right\}
\end{eqnarray*}

and thus we get%
\begin{eqnarray*}
\sum\limits_{k=0}^{n-1}\mathcal{F}_{k}\left( a,b\right) &=&\frac{\left\{ 
\begin{array}{c}
\mathcal{F}_{n+1}\left( a,b\right) +\mathcal{F}_{n}\left( a,b\right) -%
\mathcal{F}_{n-1}\left( a,b\right) -\mathcal{F}_{n-2}\left( a,b\right) \\ 
\text{ \ \ \ }-a\mathcal{F}_{1}\left( a,b\right) +ab\mathcal{F}_{0}\left(
a,b\right) -b\mathcal{F}_{0}\left( a,b\right)%
\end{array}%
\right\} }{ab} \\
&=&\frac{b\mathcal{F}_{n}\left( a,b\right) +a\mathcal{F}_{n-1}\left(
a,b\right) -a\mathcal{F}_{1}\left( a,b\right) +ab\mathcal{F}_{0}\left(
a,b\right) -b\mathcal{F}_{0}\left( a,b\right) }{ab}.
\end{eqnarray*}

Similarly, for odd $n$, we obtain%
\begin{eqnarray*}
\sum\limits_{k=0}^{n-1}\mathcal{F}_{k}\left( a,b\right)
&=&\sum\limits_{k=0}^{\frac{n-1}{2}}\mathcal{F}_{2k}\left( a,b\right)
+\sum\limits_{k=0}^{\frac{n-3}{2}}\mathcal{F}_{2k+1}\left( a,b\right) \\
&=&\frac{a\mathcal{F}_{n}\left( a,b\right) +b\mathcal{F}_{n-1}\left(
a,b\right) -a\mathcal{F}_{1}\left( a,b\right) +ab\mathcal{F}_{0}\left(
a,b\right) -b\mathcal{F}_{0}\left( a,b\right) }{ab}.
\end{eqnarray*}

As a result, we find%
\begin{equation}
\sum\limits_{k=0}^{n-1}\mathcal{F}_{k}\left( a,b\right) =\dfrac{\left\{ 
\begin{array}{c}
a^{\varepsilon \left( n\right) }b^{1-\varepsilon \left( n\right) }\mathcal{F}%
_{n}\left( a,b\right) +a^{1-\varepsilon \left( n\right) }b^{\varepsilon
\left( n\right) }\mathcal{F}_{n-1}\left( a,b\right) \\ 
\text{ \ \ }-a\mathcal{F}_{1}\left( a,b\right) +ab\mathcal{F}_{0}\left(
a,b\right) -b\mathcal{F}_{0}\left( a,b\right)%
\end{array}%
\right\} }{ab}.  \notag
\end{equation}

This completes the proof.
\end{proof}

It is clear that the following result is correct for the bi-periodic
Fibonacci matrix sequence as a consequence of the condition $(ii)$ of
Theorem $\ref{teo5}$.

\begin{corollary}
\label{cor2} For $k>0$, we have 
\begin{equation*}
\sum\limits_{k=0}^{\infty }\mathcal{F}_{k}\left( a,b\right) x^{-k}=\frac{x}{%
1-(ab+2)x^{2}+x^{4}}\left[ 
\begin{array}{cc}
x^{3}+bx^{2}-x & \frac{b}{a}x^{2}+bx-\frac{b}{a} \\ 
x^{2}+ax-1 & x^{3}-\left( ab+1\right) x+b%
\end{array}%
\right] .
\end{equation*}
\end{corollary}

\section*{Conclusion}

\qquad In this paper, we define bi-periodic matrix sequence and give some
properties of this new sequence. Thus, it is obtained a new genaralization
for the matrix sequences and number sequences that have the similar
recurrence relation in the literature. By taking into account this
generalized matrix sequence and its properties, it also can be obtained
properties of bi-periodic Fibonacci numbers. That is, if we compare the $2$%
\textit{nd} row and $1$\textit{st} column entries of obtained equalities for
matrix sequence in Section 2, we can get some properties for bi-periodic
Fibonacci numbers. Also, some well-known matrix sequences, such as
Fibonacci, Pell and $k$-Fibonacci are special cases of \{$\mathcal{F}%
_{n}\left( a,b\right) $\} matrix sequence. That is, if we choose the
different values of $a$ and $b$, then we obtain the summations, generating
functions, Binet formulas of the well-known matrix sequence in the
literature:

\begin{itemize}
\item If we replace $a=b=1$ in $\mathcal{F}_{n}\left( a,b\right) $, we
obtain the generating function, Binet formula and summations for Fibonacci
matrix sequence and Fibonacci numbers.

\item If we replace $a=b=2$ in $\mathcal{F}_{n}\left( a,b\right) $, we
obtain the generating function, Binet formula and summations for Pell matrix
sequence and Pell numbers.

\item If we replace $a=b=k$ in $\mathcal{F}_{n}\left( a,b\right) $, we
obtain the generating function, Binet formula and summations for $k$%
-Fibonacci matrix sequence and $k$-Fibonacci numbers.
\end{itemize}

\section*{Acknowledgement}

\qquad This study is a part of Arzu Co\c{s}kun's Ph.D. Thesis. Thank to the
editor and reviewers for their interests and valuable comments.

\end{document}